\documentclass{amsart}

\usepackage{amsthm}
\usepackage{amsmath}
\usepackage[mathscr]{eucal}
\usepackage{amssymb}
\usepackage{color}
\newtheorem{thm}{Theorem}[section]
\newtheorem*{maintheorem}{Main Theorem}
\newtheorem{cor}[thm]{Corollary}
\newtheorem{lem}[thm]{Lemma}
\newtheorem{prop}[thm]{Proposition}
\theoremstyle{definition}

\newtheorem{defn}[thm]{Definition}
\theoremstyle{remark}

\numberwithin{equation}{section}

\begin{document}
\bibliographystyle{plain}

\title[Definable Envelopes in $\mathfrak{M}_C$ Groups]{Definable Envelopes of Nilpotent Subgroups of Groups with
Chain Conditions on Centralizers}%
\author[Alt\i nel]{Tuna Alt\i nel}
\author[Baginski]{Paul Baginski}
\address{Institut Camille Jordan, Universit\'e Claude Bernard Lyon 1, Lyon, France 69622}
\address{Institut Camille Jordan, Universit\'e Claude Bernard Lyon 1, Lyon, France 69622}
\date{\today}

\begin{abstract}
An $\mathfrak{M}_C$ group is a group in which all chains of centralizers have finite length. 
%The $\mathfrak{M}_C$ property is possessed by many classical 
%classes of groups, and it is also implied by 
%several model-theoretic properties, such as stability. 
In this article, we show that every nilpotent subgroup of 
an $\mathfrak{M}_C$ group is contained in a definable 
subgroup which is nilpotent of the same nilpotence class. 
Definitions are uniform when the lengths of chains are bounded.
%Such definable subgroups, known as definable envelopes, were known to exist for stable groups and have been used extensively in that setting. Our results demonstrate that the existence of these envelopes derives purely from the $\mathfrak{M}_C$ property, which is not a first-order property of groups, unlike stability.
\end{abstract}


\maketitle

\section{Introduction}

Chain conditions have played a central role in modern infinite group theory and one of the most natural chain conditions is one on centralizers. A group is said to be $\mathfrak{M}_C$ if all chains of centralizers of arbitrary subsets are finite. If there is a uniform bound $d$ on the lengths of such chains, then $G$ has {\bf finite centralizer dimension (fcd)} and the least such bound $d$ is known as the {\bf $c$-dimension} of $G$, which we denote $\dim(G)$.

The $\mathfrak{M}_C$ property has been studied by group theorists since many natural classes of groups possess this property. See \cite{Bryant} for a classic paper on the properties of $\mathfrak{M}_C$ groups. Many groups possess the stronger property of fcd, including abelian groups, free groups, linear groups, and torsion-free hyperbolic groups. Khukhro's article on the solvability properties of torsion fcd groups \cite{Khukhro} compiles a lengthy list of groups with fcd. Khukhro's article, as well as several other foundational papers (see, for example, \cite{Bludov, DeWa, Khukhro, MyaShu, Wagner2}), have demonstrated that $\mathfrak{M}_C$ groups and groups with fcd are fairly well-behaved, for example by having Engel conditions closely linked to nilpotence.

For model theorists, the interest in these groups derives from the well-studied model-theoretic property of stability. A stable group must have fcd; in fact, it possesses uniform chain conditions on all uniformly definable families of subsets. Stable groups have an extensive literature in model theory (see \cite{Poizat} or \cite{Wagner}), however the properties of $\mathfrak{M}_C$ and fcd are appearing in other areas of the model theory of groups, such as rosy groups with NIP \cite{EKP} or P\i nar U\u{g}urlu's recent work on pseudofinite groups with fcd \cite{Pinar}.

The results of this paper reinforce Wagner's work \cite{DeWa, Wagner, Wagner2} in showing that several basic properties of (sub)stable groups derive purely from these simple group-theoretic chain conditions, which force the left-Engel elements
to be well-behaved. 
In contrast to Wagner's generalizations, which revealed 
that $\mathfrak{M}_C$ sufficed for many group-theoretic properties of stable groups, we shall show that $\mathfrak{M}_C$ suffices for a {\em logical} property of stable groups, asserting the existence of certain definable groups. 
In the end, our results will yield an alternative path to a conclusion
that also follows from Wagner's analysis of left-Engel elements.

It has been known for some time \cite[Theorem 3.17]{Poizat}, that if $G$ is a stable group and $H$ is a nilpotent (or solvable) subgroup, then there exists a definable subgroup $d(H)$ of $G$ which contains $H$ and has the same nilpotence class (derived length) as $H$. Such a subgroup $d(H)$ is called a {\bf definable envelope} of $H$. The existence of definable envelopes allowed logicians to approximate arbitrary nilpotent subgroups of stable groups with slightly ``larger'' nilpotent subgroups which were definable, i.e. manipulable with model-theoretic techniques.

Our main theorem asserts the existence of definable envelopes of nilpotent subgroups in $\mathfrak{M}_C$ groups and uniformly definable envelopes for groups with fcd. Definability here always refers to formulas in the language $\mathcal{L}_G$ of groups. These envelopes are $N_G(H)$-normal, meaning that if an element normalizes $H$, it also normalizes the envelope.

\begin{maintheorem}
Let $G$ be an $\mathfrak{M}_C$ group and $H\leq G$ a nilpotent subgroup. Then there exists a subgroup $D\leq G$ definable in the language of groups with parameters from $G$, such that $H\leq D$, $D$ is $N_G(H)$-normal, and $D$ is nilpotent of the same nilpotence class as $H$. 

Moreover, for every pair of positive integers $d$ and $n$, there exists a formula $\phi_{d,n}(x, \overline{y})$, where $\ell(\overline{y})=dn$, such that for any group $G$ of dimension $d$ and any $H\leq G$ nilpotent of class $n$, there exists a tuple $\overline{a}\in G$ such that $\phi_{d,n}(G,\overline{a})$ is a nilpotent subgroup of $G$ of class $n$ which contains $H$ and is $N_G(H)$-normal.
\end{maintheorem}

We hope that this result will prove useful in some of the current areas in logic where $\mathcal{M}_C$ groups are appearing. 
This result may also open to the door to studying some of the logical properties 
of the non-elementary classes of groups with fcd listed in \cite{Khukhro}.
It is worth mentioning that \cite{Aldama} and \cite{Milliet} contain
results on definable envelopes in elementary classes of
groups whose theories are NIP or simple, respectively.

We assume only a rudimentary knowledge of model theory and logic, namely the notions of ``definability'' and the Compactness Theorem. Readers may consult any introductory text, such as \cite{Hodges} or \cite{Poizat}, for explanations of these notions. Otherwise, the material will be primarily group-theoretic and self-contained.

In the next section, we will define relevant terms from group theory and prove some fundamental lemmas about groups in general. In the following section, we restrict our focus to $\mathfrak{M}_C$ groups and prove our main theorem and some corollaries.

\section{Preliminaries}

We write $A\leq G$ to denote that $A$ is a subgroup of $G$ and $A\lhd G$ to denote $A$ is normal in $G$. If $A\subseteq G$ then $\langle A\rangle$ denotes the subgroup generated by $A$. For any subset $A$ of $G$, the centralizer of $A$ is $C_G(A)=\{g\in G\,|\, \forall a\in A\,\, ga=ag\}$, while the normalizer of $A$ is $N_G(A)=\{g\in G\,|\, \forall a\in A\,\, g^{-1}ag\in A\}$. If $A$ and $B$ are subgroups of a group $G$, then $A$ is $N_G(B)$-normal if $N_G(B)\leq N_G(A)$.

Given $g,h\in G$, the commutator of $g$ and $h$ is $[g,h]:=g^{-1}h^{-1}gh$. Iterated commutators are interpreted as left-normed, i.e., $[x,y,z]$ will denote $[[x,y], z]$. When $A, B\subseteq G$, then we write $[A,B]:=\langle \{ [a,b]\,|\, a\in A, \, b\in B\}\rangle$. We define the lower central series of $G$ as $\gamma_1(G):=G$ and $\gamma_{k+1}(G):=[\gamma_k(G), G]$. A group $G$ is nilpotent if $\gamma_n(G)=1$ for some $n<\omega$; the least $n\geq 0$ for which $\gamma_{n+1}(G)=1$ is the nilpotence class of $G$. It is clear that a subgroup of a nilpotent group is nilpotent of equal or lesser nilpotence class.

The Hall-Witt identity relates the commutators of three elements: For all $x,y,z\in G$, 
\begin{equation}\label{Witt}
1=[x,y^{-1}, z]^y[y,z^{-1},x]^z[z,x^{-1},y]^x=[x,y,z^x][z,x,y^z][y,z,x^y]
\end{equation}
The Hall-Witt identity is used to prove the well-known 
Three Subgroup Lemma, which we state in the needed level of generality. 

\begin{lem}{\bf \cite[Three Subgroup Lemma, 5.1.10]{Robinson}}\label{3subgp}
Let $G$ be a group, $N$ a subgroup, and $K,L,$ and $M$ subgroups of $N_G(N)$. Then $[K,L,M]\leq N$ and $[L,M,K]\leq N$ together imply $[M,K,L]\leq N$.
\end{lem}

This article shall be concerned with chains of centralizers. However, in order to analyze them fully, we shall need a more general definition of {\em iterated} centralizers.

\begin{defn}
Let $P$ be a subgroup of $G$. We define the {\bf iterated centralizers of} $P$ {\bf in} $G$ as follows. Set $C_G^0(P)=1$ and for $n\geq 1$, let
\begin{equation*}
C_G^n(P)=\left.\left\{x\in \bigcap_{k<n} N_G(C_G^k(P))\, \right|\, [x, P]\subseteq C_G^{n-1}(P)\right\}
\end{equation*}
When $P=G$, the $n$th iterated centralizer of $G$ is more commonly known as $Z_n(G)$, the {\bf $n$th center} of $G$.
\end{defn}

The groups $Z_n(G)=C_G^n(G)$ are all characteristic in $G$, so that their definition simplifies to $Z_0(G)=\{1\}$ and $Z_{n+1}(G)=\{g\in G\,|\, [g, G]\subseteq Z_n(G)\}$ for all $n\geq 0$. The subgroup $Z_1(G)=Z(G)$ is the center of $G$. This series is known as the upper central series; a group is nilpotent of class $n$ if and only if $Z_n(G)=G$.

It is easy to show  that $P$ normalizes each $C_G^n(P)$ and consequently that each $C_G^n(P)$ is a subgroup of $G$. Furthermore, the intersections with $P$ are well-behaved: $C_G^n(P)\cap P=Z_n(P)$. If $P$ is a nilpotent subgroup of $G$ of nilpotence class $n$, then $P\leq C_G^n(P)$. These results may all be proven easily by induction, as can the following lemma due to P. Hall which relates iterated centralizers of $P$ to the lower central series of $P$.

\begin{lem}{\bf \cite[Satz III.2.8]{Huppert}} \label{huppertfact}
Let $G$ be a group and $P$ a subgroup of $G$. Then
\[
[\gamma_i(P),C_G^k(P)]\ \leq\ C_G^{k-i}(P)
\]
for all positive integers $i$ and $k$ such that $i\leq k$. In particular,
\[
[\gamma_i(G), Z_k(G)]\leq Z_{k-i}(G).
\]
\end{lem}

Bryant (Lemma 2.5 in \cite{Bryant}) used Hall's lemma to determine 
conditions under which one could conclude a group and 
a subgroup have the same iterated centralizer. We shall pursue 
the same goal and restructure Bryant's argument for our purposes. 
Th following technical lemma %argument 
is the heart of the proof of our main theorem.
Its proof almost reproduces Bryant's subtle argument.
We include it not only for completeness, but also to clarify how our lemma and Bryant's relate to each other, despite statements that differ considerably.

\begin{lem} \label{Bryantlemma}
Let $k\geq 1$ be an integer, $G$ be a group, and $X\leq P$ be two subgroups of $G$
satisfying the following conditions:
\begin{enumerate}
\item $C_G^i(X)=C_G^i(P)$ for all $i\in\{0,\ldots,k-1\}$;
\item $[\gamma_k(P),C_G^k(X)]=1$; %;
\item $C_G(X)=C_G(P)$.
\end{enumerate}
Then $C_G^k(X)=C_G^k(P)$.
\end{lem}
\begin{proof}
The argument %follows Lemma 2.5 of \cite{Bryant} and 
proceeds by induction on $k$, with $k=1$ given by hypothesis $(3)$. So we assume $k>1$. 

By hypothesis $(1)$, $C_G^i(X)=C_G^i(P)$ for all $0\leq i\leq k-1$, so set $C_i=C_G^i(X)$. We claim that the normalizers of these $C_i$ contain four groups important to this proof. Namely, we claim that for all $1\leq i\leq k-1$, 
\[
(A)\ \ \ \ \ 
X \ \cup \ \gamma_{k-i}(P) \ \cup \ C_G^k(X) \ \cup \ C_G^k(P) \ \subseteq \ \bigcap_{j=0}^{k-1} N_G(C_j)\ .
\]
Since $X$ normalizes all its iterated centers, we find $X\leq N_G(C_j)$ for all $j\leq k-1$. Similarly, since $P$ normalizes all its iterated centers, we find $\gamma_{k-i}(P)\leq P\leq N_G(C_j)$ for all $j\leq k-1$. Lastly, by the definition of iterated centralizers, $C_G^k(X)\leq N_G(C_j)$ and $C_G^k(P)\leq N_G(C_j)$ for all $j\leq k-1$.

We next claim as in Bryant's proof that
\[
(B)\ \ \ \ \ 
[\gamma_{k-i}(P),C_G^k(X)]\ \leq\ C_i\ \mbox{ for }\ i=0,1,\ldots,k-1\ .
\]
We shall prove $(B)$ by induction on $i$ for $0\leq i \leq k-1$. 
Note that $i=0$ is precisely hypothesis $(2)$, so assume $i\geq 1$.  
We shall prove $(B)$ using the Three Subgroup Lemma (Lemma \ref{3subgp}) 
with $\gamma_{k-i}(P), X,$ and $C_G^k(X)$ relative to the group $C_G^{i-1}(P)=C_{i-1}$. 
By $(A)$, %we have already shown 
these three groups normalize $C_{i-1}$. Since $X\leq P$, we have by Lemma \ref{huppertfact} 
and by induction on $i$ that
\[
[\gamma_{k-i}(P),X,C_G^k(X)]\ \leq\ [\gamma_{k-i+1}(P),C_G^k(X)]\ \leq\ C_G^{i-1}(P) \ .
\]
By the definition of iterated centralizers, $[X,C_G^k(X)]\leq C_G^{k-1}(X)=C_{k-1}$ and thus we obtain the following chain of inequalities:
\[
[X,C_G^k(X),\gamma_{k-i}(P)]\ \leq\ [C_{k-1},\gamma_{k-i}(P)]\ 
= [C_G^{k-1}(P),\gamma_{k-i}(P)]\ \leq\ C_G^{i-1}(P) \ .
\]
The last inequality follows from Lemma \ref{huppertfact}. By the Three Subgroup Lemma (Lemma \ref{3subgp}), we conclude that 
\[
[C_G^k(X),\gamma_{k-i}(P),X]\ \leq\ C_G^{i-1}(P) = C_{i-1} = C_G^{i-1}(X) \ .
\]
Since $[C_G^k(X),\gamma_{k-i}(P)] \leq N_G(C_j)$ for all $j\leq i-1$ by $(A)$, we conclude from the definition of iterated centralizers that $[C_G^k(X),\gamma_{k-i}(P)]\leq C_G^i(X)=C_i$, yielding our claim $(B)$.

Setting $i=k-1$ in $(B)$ gives
\[
[P, C_G^k(X)]\leq C_{k-1}=C_G^{k-1}(P) \ .
\]
Yet $(A)$ with $i=k-1$ implies that $C_G^k(X)\leq N_G(C_G^j(P))$ for all $j\leq k-1$, so by the definition of iterated centralizers, $C_G^k(X)\leq C_G^k(P)$. On the other hand, since $X\leq P$, we find:
\[
[X, C_G^k(P)]\leq [P,C_G^k(P)]\leq C_G^{k-1}(P)=C_{k-1}=C_G^{k-1}(X)
\]
Again, by $(A)$ with $i=k-1$, we find that $C_G^k(P)\leq N_G(C_G^j(X))$ for all $j\leq k-1$. So by definition of iterated centralizers, $C_G^k(P)\leq C_G^k(X)$ and we have equality.
\end{proof}

We shall also need a lemma relating the iterated centralizers of three nested groups.

\begin{lem}\label{threeiterated}
Let $A\leq B\leq C$ be groups and suppose that for all $k<n$, $C_C^k(A)=C_C^k(C)$. Then $C_B^j(A)=C_C^j(A)\cap B$ for all $j\leq n$. 
\end{lem}
\begin{proof}
We shall prove this lemma by induction on $1\leq j\leq n$. For $j=1$, this is just the statement that $C_B(A)=C_C(A)\cap B$. Assume the claim is true for some $j<n$. Then $C_C^\ell(A)=C_C^\ell(C)=Z_\ell(C)$ for all $\ell\leq j$ and thus by the induction hypothesis, we have for all $\ell\leq j$ that $N_B(C_B^\ell (A))=N_B(C_C^\ell(A)\cap~\!\!B)=N_B(Z_\ell(C)\cap B)$. Because $Z_\ell(C)$ is characteristic in $C$, $N_B(Z_\ell(C)\cap B)=B$. Thus, $x\in C_B^{j+1}(A)$ if and only if $x\in B$ and $[x, A]\subseteq C_B^j(A)=C_C^j(A)\cap B$, with the latter equality coming from the induction hypothesis on $j$. For any $x\in C_C^{j+1}(A)$, we have by definition that $[x,A]\in C_C^j(A)$. Since $A\leq B$, we find $C_C^{j+1}(A)\cap B\subseteq C_B^{j+1}(A)$. On the other hand, $C_B^{j+1}(A)\leq C=\bigcap_{\ell=1}^j N_C(C_C^\ell(C))=\bigcap_{\ell=1}^j N_C(C_C^\ell(A))$ and $[C_B^{j+1}(A), A]\leq C_B^j(A)=C_C^j(A)\cap B$. Thus $C_B^{j+1}(A)\subseteq C_C^{j+1}(A)$ by definition and we have $C_B^{j+1}(A)=C_C^{j+1}(A)\cap B$. By induction, $C_B^j(A)=C_C^j(A)\cap B$ for all $j\leq n$.
\end{proof}

\section{Bounded chains of centralizers and definable envelopes}\label{envelopes}

As mentioned in the introduction, several well-studied classes of groups in group theory and model theory possess chain conditions on their centralizers. We restate the definitions of these chain conditions precisely.

\begin{defn}
A group $G$ is has the {\bf chain condition on centralizers}, denoted $\mathfrak{M}_C$, if there exists no infinite sequence of subsets $A_n\subseteq G$ such that $C_G(A_n)>C_G(A_{n+1})$ for all $n<\omega$. 

A group $G$ has {\bf finite centralizer dimension (fcd)} if there is a uniform bound $n\geq 1$ on any chain $G=C_G(1)> C_G(A_1) > \ldots > C_G(A_n)$ of centralizers of subsets $A_i$ of $G$. The least bound (i.e. the length of the longest chain of centralizers) is known as the {\bf $c$-dimension} of $G$, which we will denote $\dim(G)$.
\end{defn}

Note that since $C_G(C_G(C_G(A)))=C_G(A)$ for all $A\subseteq G$, all descending chains of centralizers are finite if and only if all ascending chains are finite.  An immediate consequence of the finite chain condition is the following observation:

\begin{prop}\label{fcdfiniteintersection}
Let $G$ be an $\mathfrak{M}_C$ group. If $A\subseteq G$, then there is an $A'\subseteq A$ finite such that $C_G(A)=C_G(A')$. If $G$ has centralizer dimension $d$, then $A'$ can be chosen such that $|A'|\leq d$.
\end{prop}

If $\dim(G)=1$, then for any $g\in G$, we have $G=C_G(1)\geq C_G(g)\geq C_G(G)=Z(G)$, with at most one of these inequalities strict. Thus $g\in Z(G)$ and $G$ is abelian.  Since centralizers relative to a subgroup $H$ of $G$ are calculated by simply intersecting with $H$, it is clear that a subgroup of an $\mathfrak{M}_C$ group is $\mathfrak{M}_C$. Furthermore, if $G$ has fcd and $H\leq G$, then $H$ has fcd and $\dim(H)\leq \dim(G)$. Clearly if $G$ has fcd and $H\leq C_G(A)$ for some $A\subseteq G$ with $A\not\subseteq Z(G)$, then $\dim(H)<\dim(G)$. In the next lemma, valid for all groups, we revise this critical observation with necessary conditions for a subgroup to be contained in a centralizer of a noncentral subset.

\begin{prop}\label{bottomchain}
Let $G$ be a group and $H\leq G$. Then one of the following is true:
\begin{enumerate}
\item $H\leq Z(G)$; 
\item there exists a subset $A\subseteq G$ such that $H\leq C_G(A)<G$; or
\item $C_G(H)=Z(G)$, and hence $Z(H)=Z(G)\cap H$, i.e. $Z(H)\leq Z(G)$.
\end{enumerate}
\end{prop}
\begin{proof}
Assume $H\not\leq Z(G)$, so $C_G(H)<G$. If $C_G(H)> Z(G)$, then $H\leq C_G(C_G(H))<G$, so $A=C_G(H)$ witnesses $(2)$. Thus, if $(2)$ does not hold for $H$, then $C_G(H)=Z(G)$ and so clearly $Z(H)=C_G(H)\cap H=Z(G)\cap H$ and $Z(H)\leq Z(G)$.
\end{proof}

While both $\mathfrak{M}_C$ and fcd are preserved 
under subgroups and finite direct products \cite{MyaShu}, 
they behave poorly under quotients. 
The quotient of an $\mathfrak{M}_C$ group, even by its center, 
may fail to be $\mathfrak{M}_C$ (see \cite{Bryant}). 
This is the principal complication in the proof
of our main theorem. 
%; this group is $\mathfrak{M}_C$ but does not have fcd). %\commentb{similar result for fcd?}

From a logical perspective, the class of $\mathfrak{M}_C$ groups is not elementary: indeed, it is consistent to have chains of arbitrarily long length, so by the Compactness Theorem there would be an elementary extension of an $\mathfrak{M}_C$ group with an infinite chain of centralizers; this extension is clearly not $\mathfrak{M}_C$. Conversely, if all groups elementarily equivalent to a group $G$ are $\mathfrak{M}_C$, then there must be a uniform bound on the lengths of chains of centralizers in $G$, i.e., $G$ has fcd. Indeed, having centralizer dimension $d$ can be expressed by a first-order formula in the language $\mathcal{L}_G$ of groups: $\not\!\exists\, x_1, \ldots , x_{d+1}, \, C(1)>C(x_1)> C(x_1, x_2) > \ldots > C(x_1, x_2, \ldots, x_{d+1})$. Therefore, groups with fcd have the advantage of being analyzable using model theoretic methods since for a fixed dimension $d$ they form an elementary class. Many classic families of groups in model theory, such as stable groups, have fcd and are not simply $\mathfrak{M}_C$.

Our goal in this section is to demonstrate the existence of definable envelopes of nilpotent subgroups of $\mathfrak{M}_C$ groups. Roughly speaking, if $G$ is an $\mathfrak{M}_C$ group and $H$ is a nilpotent subgroup, we show that there is a definable nilpotent group $D$ which contains $H$ but is not much ``larger'', in the sense that it has the same nilpotence class. In model theory, the existence of such envelopes allows one to replace an arbitrary nilpotent group with a similar one which can be manipulated using model theoretic techniques. While one might perhaps expect such envelopes to exist in the elementary class of groups with a fixed centralizer dimension, we have succeeded in showing such envelopes exist even for the non-elementary class of $\mathfrak{M}_C$ groups. The advantage of fcd in this case is uniform definability of these envelopes in terms of the dimension of the ambient group and the nilpotence class of the subgroup. We are now ready to prove our main theorem.

\begin{thm}\label{defenv}
Let $G$ be an $\mathfrak{M}_C$ group and $H\leq G$ a nilpotent subgroup. Then there exists a subgroup $D\leq G$ definable in the language of groups with parameters from $G$, such that $H\leq D$, $D$ is $N_G(H)$-normal, and $D$ is nilpotent of the same nilpotence class as $H$. 

Moreover, for every pair of positive integers $d$ and $n$, there exists a formula $\phi_{d,n}(x, \overline{y})$, where $\ell(\overline{y})=dn$, such that for any group $G$ of dimension $d$ and any $H\leq G$ nilpotent of class $n$, there exists a tuple $\overline{a}\in G$ such that $\phi_{d,n}(G,\overline{a})$ is a nilpotent subgroup of $G$ of class $n$ which contains $H$ and is $N_G(H)$-normal.
\end{thm}
\begin{proof}
Let $G$ be a $\mathfrak{M}_C$ group and $H$ be a nilpotent subgroup of $G$ of class $n$. For all $1\leq k\leq n$, we will construct a descending chain of definable subgroups $(E_k)_{k=1}^n$ having the following properties:
\begin{enumerate}
\item each $E_k$ is definable;
\item each $E_k$ contains $H$;
\item for each $j\in\{1,\ldots,k\}$ and each subgroup $H\leq P\leq E_k$,
\[
\ \ C_{E_k}^j(H)\ \ =\ \ C_{E_k}^j(P)\ \ =\ \ C_{E_k}^j(E_k)\ \ =\ \ Z_j(E_k)\ \ ;
\]
\item each $E_k$ is $N_G(H)$-normal.
\end{enumerate}

By the chain condition, the collection of all centralizers of $G$ 
which contain $H$ has a least element, $C_G(A)$, which must be $N_G(H)$-normal. 
By Proposition \ref{fcdfiniteintersection}, $A$ can be taken to 
be a finite set $A_1=\{x_{1,1}, \ldots , x_{1, m_1}\}$, with $m_1\leq \dim(G)$ 
if $G$ has fcd. Set $E_1=C_G(A)$. 
Since $C_G(A)$ is $N_G(H)$-normal, $N_G(H)\leq N_G(HZ(C_G(A)))$. On the other
hand, any $N_G(HZ(C_G(A)))$-conjugate of $E_1$ still contains $H$.
By the minimal choice of $C_G(A)$, we conclude that
$E_1=C_G(A)$ is $N_G(HZ(C_G(A)))$-normal. Thus,
without loss of generality we may replace $H$ with $HZ(C_G(A))$. 
By Proposition \ref{bottomchain}, we conclude $C_{E_1}(H)=Z(H)=Z(E_1)=C_{E_1}(E_1)$.  
In fact, by this proposition, for any group $P$ with $H\leq P\leq E_1$, 
we have $C_P(H)=Z(H)=Z(P)=C_{E_1}(P)$. So we have successfully constructed $E_1$.

Suppose now that the $E_j$ have been constructed for $1\leq j<k$ and 
consider $P$ such that $H\leq P\leq E_{k-1}$. By property (3) and Lemma \ref{threeiterated}, we have 
\[
\ \ C_P^j(H)\ = \ C_{E_{k-1}}^j(H)\cap P \ \ \ (\bullet )
\]
for all $1\leq j\leq k$.
%Note that the third property implies $N_{E_{k-1}}(C_{E_{k-1}}^j(H))=N_{E_{k-1}}(Z_j(E_{k-1}))=E_{k-1}$ for all $1\leq j<k$. If $H\leq P\leq E_{k-1}$, we will now show by induction that for $1\leq j\leq k$,
%\[
%\ \ C_P^j(H)\ = \ C_{E_{k-1}}^j(H)\cap P\ . \ \ \ (\bullet )
%\]
%For $j=1$, this is just the statement that $C_P(H)=C_{E_{k-1}}(H)\cap P$. Assume the statement is true for $j<k$. Then by the induction hypothesis
%on $j$, for all $\ell\leq j$, $N_P(C_P^\ell (H))=N_P(C_{E_{k-1}}^\ell(H)\cap P)=N_P(Z_\ell(E_{k-1})\cap P)$, 
%by property $(3)$ above. Since $Z_\ell(E_{k-1})$ is characteristic in $E_{k-1}$, 
%$N_P(Z_\ell(E_{k-1})\cap P)=P$. Therefore, $x\in C_P^{j+1}(H)$ if and only if 
%$[x, H]\subseteq C_P^j(H)=C_{E_{k-1}}^j(H)\cap P$, with the latter equality coming from the induction hypothesis
%on $j$. Since $H\leq P$, we see $C_{E_{k-1}}^{j+1}(H)\cap P\subseteq C_P^{j+1}(H)$. 
%On the other hand, $C_P^{j+1}(H)\leq E_{k-1}=\bigcap_{\ell=1}^j N_{E_{k-1}}(C_{E_{k-1}}^\ell(H))$ and $[C_P^{j+1}(H), H]\subseteq C_P^j(H)=C_{E_{k-1}}^j(H)\cap P$, thus $C_P^{j+1}(H)\subseteq C_{E_{k-1}}^{j+1}(H)$ and we have $C_P^{j+1}(H)=C_{E_{k-1}}^{j+1}(H)\cap P$.
Before defining $E_k$, we first define an intermediate definable subgroup for every
$h\in C_{E_{k-1}}^k(H)$:
\[
\ \ E_k(h)\ =\ \{\ x\in E_{k-1}\ |\ [x,h]\in Z_{k-1}(E_{k-1})\ \}\ .
\]
Since $Z_{k-1}(E_{k-1})\lhd E_{k-1}$, $E_k(h)$ is a subgroup of $E_{k-1}$. It is definable in $E_{k-1}$ with $h$ as parameter. Since by induction on $k$, $C_{E_{k-1}}^{k-1}(H)=Z_{k-1}(E_{k-1})$ and $H\leq E_{k-1}$, $H\leq E_k(h)$. Clearly, $h\in E_k(h)$ as well.

Since $h\in E_k(h)$, we find that $[E_k(h),h]\leq Z_{k-1}(E_{k-1})\cap E_k(h)\leq Z_{k-1}(E_k(h))$. As a result, $h\in Z_k(E_k(h))$. By Lemma \ref{huppertfact},  $[\gamma_k(E_k(h)), Z_k(E_k(h))]=1$, we conclude that
\[
\ \ [\gamma_k(E_k(h)),h]\ =\ 1\ .\ \ \ (\bullet \bullet)
\]
By the $\mathfrak{M}_C$ condition on the ambient group,  there exist $x_{k,1},\ldots,x_{k,m_k}\in C_{E_{k-1}}^k(H)$ for some  $m_k\in\mathbb{N}^*$ such that  $C_{E_{k-1}}(C_{E_{k-1}}^k(H))=C_{E_{k-1}}(x_{k,1},\ldots,x_{k,m_k})$. We then define
\[
\ \ E_k\ =\ \bigcap_{i=1}^{m_k} E_k(x_{k,i})\ .
\]
The subgroup $E_k$ is definable in $E_{k-1}$ with parameters $x_{k,1},\ldots, x_{k,m_k}$ and contains $H$. If $G$ has dimension $d$, then we may take $m_k=d$, so this definition of $E_k$ in terms of $E_{k-1}$ is uniform in terms of $d$ parameters. Moreover, $H\leq E_k\leq E_{k-1}$. By the choice of the $x_i$, $(\bullet \bullet)$ implies that 
\[
[\gamma_k(E_k), C_{E_{k-1}}^k(H)]=1\,. \ \ (\bullet \bullet \bullet)
\]
By property (3) of $E_{k-1}$, we have $C_{E_{k-1}}^i(H)=C_{E_{k-1}}^i(E_k)$ 
for all $0\leq i\leq k-1$. In particular (as $k\geq 2$), 
$C_{E_{k-1}}(H)=Z(E_{k-1})=C_{E_{k-1}}(E_k)$. 
By Lemma \ref{Bryantlemma}, $C_{E_{k-1}}^k(H)=C_{E_{k-1}}^k(E_k)$. 
Therefore, for all $h\in C_{E_{k-1}}^k(H)$, $[E_k, h]\subseteq C_{E_{k-1}}^{k-1}(E_k)=Z_{k-1}(E_{k-1})$ 
by property (3) of $E_{k-1}$. In other words, $E_k\leq E_k(h)$ for all $h\in C_{E_{k-1}}^k(H)$ and hence 
\[
E_k=\bigcap_{h\in C_{E_{k-1}}^k(H)} E_k(h)\,.
\]
By property (4) of $E_{k-1}$, we conclude that $C_{E_{k-1}}^k(H)$ 
and $Z_{k-1}(E_{k-1})$ are $N_G(H)$-normal and thus so is $E_k$, establishing the fourth property for $E_k$.

We are left with establishing the third property for $E_k$. Let $P$ be any subgroup such that $H\leq P\leq E_k\leq E_{k-1}$. The statement $(\bullet)$ demonstrates that $C_{E_k}^j(H)\ = \ C_{E_{k-1}}^j(H)\cap E_k$ for all $1\leq j\leq k$. Together with $(\bullet \bullet \bullet)$, we conclude that $[\gamma_k(E_k), C_{E_k}^k(H)]=1$, and hence $[\gamma_k(P), C_{E_k}^k(H)]=1$, which is condition $(2)$ in Lemma \ref{Bryantlemma}. Condition $(3)$ of this lemma follows from the fact in the first paragraph that $C_{E_k}(E_k)=C_{E_k}(P)=C_{E_k}(H)=Z(E_k)=Z(H)$.

It remains to show that $C_{E_k}^j(H)=C_{E_k}^j(P)=C_{E_k}^j(E_k)=Z_j(E_k)$ for $1\leq j\leq k$ in order to complete the induction on $k$. Note that we always have $C_G^n(G)=Z_n(G)$, so the first two equalities are the only ones we need to establish. If we have succeeded in showing $C_{E_k}^j(H)=C_{E_k}^j(P)=C_{E_k}^j(E_k)=Z_j(E_k)$ for all $1\leq j<k$, then Lemma \ref{Bryantlemma} guarantees that $C_{E_k}^k(H)=C_{E_k}^k(P)=C_{E_k}^k(E_k)$. So it remains for us to show that $C_{E_k}^j(H)=C_{E_k}^j(P)=C_{E_k}^j(E_k)=Z_j(E_k)$ for all $1\leq j<k$. We shall prove this by induction on $j<k$. We have already noted that for $j=1$, $C_{E_k}(E_k)=C_{E_k}(P)=C_{E_k}(H)=Z(E_k)=Z(H)$. Assume it is true for all $\ell<j$. 
We know from $(\bullet)$ and property (3) for $E_{k-1}$ that
\[
\ \ \ C_{E_k}^j(H)\ = \ C_{E_{k-1}}^j(H)\cap E_k \ = \ Z_j(E_{k-1})\cap E_k.
\]
Similarly, $C_{E_k}^j(P) = Z_j(E_{k-1})\cap E_k$. Yet $Z_j(E_{k-1})\cap E_k\leq Z_j(E_k)$, so $C_{E_k}^j(P)$ and $C_{E_k}^j(H)$ are both subgroups of $Z_j(E_k)$. By the induction hypothesis on $j-1$, 
\[
\ \ [H, Z_j(E_k)] \ \leq \ [P, Z_j(E_k)] \ \leq \ [E_k, Z_j(E_k)]\ \leq \ Z_{j-1}(E_k) = C_{E_k}^{j-1}(H) = C_{E_k}^{j-1}(P).
\]
By the induction hypothesis on $\ell<j$, we know $N_{E_k}(C_{E_k}^\ell(H))=N_{E_k}(C_{E_k}^\ell(P))=N_{E_k}(Z_\ell(E_k))=E_k$, 
so by the definition of iterated centralizers,
we determine $Z_j(E_k)\leq C_{E_k}^j(H)\cap C_{E_k}^j(P)$. Therefore $C_{E_k}^j(E_k)=Z_j(E_k)=C_{E_k}^j(P)=C_{E_k}^j(H)$, as desired.

This completes the induction on $k$ constructing the descending chain of subgroups $E_k$. The definable envelope of $H$ is $Z_k(E_k)$, where $k$ is the nilpotence class of $H$. Indeed, $H\leq C_{E_k}^k(H)=Z_k(E_k)$ by construction. As noted during the proof, if $G$ has dimension $d$, the iterative construction of the $E_k$ is uniform and requires $d$ parameters at each stage for a total of $dn$ parameters from $G$. Thus for groups with a fixed centralizer dimension, the definable envelopes of nilpotent subgroups are uniformly definable.
\end{proof}

\begin{cor}\label{normaldef}
Let $G$ be an $\mathfrak{M}_C$ group and $H\lhd G$ be a normal nilpotent subgroup. Then $H\leq D\leq G$ for some normal, definable subgroup $D$ of $G$ which is nilpotent of the same class as $H$.

For every pair of positive integers $d$ and $n$, there is a formula $\phi_{d,n}(x,\overline{y})$ such that if $G$ is a group of dimension $d$ and $H\lhd G$ is a normal nilpotent subgroup of nilpotence class $n$, then for some $\overline{a}\in G$, $\phi_{d,n}(G,\overline{a})$ is a normal nilpotent subgroup of class $n$ which contains $H$.
\end{cor}
\begin{proof}
Let $H\lhd G$ be a normal nilpotent subgroup. By Theorem \ref{defenv}, there is a definable nilpotent subgroup $D$ of $G$ which is nilpotent of the same class as $H$ such that $H\leq D\leq G$. This group $D$ is $N_G(H)$-normal, and hence normal in $G$. 
%But then $H\leq \bigcap_{g\in G} D^g$, which is still a definable nilpotent subgroup of $G$ of the same nilpotence class as $H$, however, $\bigcap_{g\in G} D^g$ is normal. If $G$ has fixed dimension $d$ and $H$ is of fixed nilpotence class $n$, then by Theorem \ref{defenv}, we can choose the envelopes $D$ to be uniformly definable; clearly, then so are the groups $\bigcap_{g\in G} D^g$.
\end{proof}

The subgroup generated by all normal nilpotent subgroups of $G$ is called the {\bf Fitting group} $F(G)$ of $G$. We shall say the {\bf definable Fitting group} $dF(G)$ of $G$ is the subgroup generated by all definable normal nilpotent subgroups of $G$. Clearly, for any group $G$ we have $dF(G)\leq F(G)$ and Corollary \ref{normaldef} indicates that in a $\mathfrak{M}_C$ group, $F(G)=dF(G)$. For completeness, we mention that Wagner \cite[Definition 1.1.7]{Wagner} defined a similar concept known as the {\em definable generalized Fitting subgroup}, $F^*(G)$. We have the following relationship between these three notions: $dF(G)\leq F(G)\leq F^*(G)$ in all groups $G$. None of these groups need be definable, though it may occur that in a particular model of $\textrm{Th}(G)$, these groups coincide with definable groups. For $dF(G)$, such a situation is often not a coincidence.

\begin{lem}\label{defFitting}
Let $G$ be a group and suppose the definable Fitting subgroup $dF(G)$ coincides with a definable nilpotent subgroup $\phi(G,\overline{a})$. Then $dF(G)$ is $\emptyset$-definable in $Th(G)$, the theory of $G$ in the language $\mathcal{L}_G$ of groups.
\end{lem}
\begin{proof}
Suppose $dF(G)=\phi(G,\overline{a})$ for some $\mathcal{L}_G$ formula $\phi(x,\overline{y})$ and $\overline{a}\in G$. Suppose $dF(G)$ has nilpotence class $n$. In any elementary extension $G'\succeq G$, $\phi(G', \overline{a})$ is a normal group of nilpotence class $n$, and thus contained in $dF(G')$. For any formula $\theta(x,\overline{y})$ in the language of groups and any integer $k\geq 1$, the following sentence is true in $G$:
\begin{multline*}
G\models \forall \overline{y} \, (\, (\,\theta(x, \overline{y}) \mbox{ is a normal nilpotent group of class $k$}\,) \rightarrow \\ \forall x\, (\, \theta(x, \overline{y})\rightarrow \phi(x,\overline{a})\,)).
\end{multline*}
Therefore the same sentence is true in $G'$, so that $\phi(G',\overline{a})$ contains all definable normal nilpotent subgroups of $G'$. By definition of the definable Fitting group, $\phi(G',\overline{a})=dF(G')$.

We now claim that $dF(G)=\phi_{d,n}(G,\overline{a})$ can be defined without parameters in $\mathcal{L}_G$, the language of groups. If not, by Svenonius' Theorem \cite[Thm 9.2]{Poizatmodel}, there is an elementary extension $(G', \overline{a}')\succeq (G, \overline{a})$ and an $\mathcal{L}_G$-automorphism $\sigma$ of $G'$ which does not preserve $\phi_{d,n}(G',\overline{a}')$. Yet $\overline{a}$ and $\overline{a}'$ have the same $\mathcal{L}$-type, so $\phi_{d,n}(G',\overline{a}')$ contains all definable normal nilpotent subgroups of $G'$, and by the same argument as above, $\phi_{d,n}(G',\overline{a}')=dF(G')$. Since $dF(G')$ is characteristic, it is preserved by all automorphisms of $G'$, a contradiction. So the definable Fitting subgroup is $\emptyset$-definable in $Th(G)$.
\end{proof}

We now address the issue of the definability of the Fitting group $F(G)$ directly. Of great importance will be the following result:

\begin{lem}[Theorem 1.2.11, \cite{Wagner}]\label{MCFittingnilpotent}
Let $G$ be $\mathfrak{M}_C$. Then the Fitting subgroup $F(G)$ is nilpotent.
\end{lem}

Using this result and the work of Bludov \cite{Bludov}, Wagner has already showed that in an $\mathfrak{M}_C$ group, $F(G)$ is $\emptyset$-definable. In fact, his result states that $F(G)=\overline{L}(G)$, the set of bounded left Engel elements (\cite[Corollary 2.5]{Wagner2}). The nilpotence of $F(G)$ (Lemma \ref{MCFittingnilpotent}), then implies that $F(G)$ consists of those elements $x\in G$ which satisfy the $n$th left Engel condition for a fixed $n$. We shall arrive at the $\emptyset$-definability of $F(G)$ by using the avenue of definable envelopes instead.

\begin{cor}\label{Fittingdef}
Let $G$ be an $\mathfrak{M}_C$ group. The Fitting subgroup $F(G)$ coincides with a $\emptyset$-definable, nilpotent subgroup of $G$. If $G$ has fcd, $F(G)$ is itself $\emptyset$-definable in $Th(G)$.
\end{cor}
\begin{proof}
Let $G$ be an $\mathfrak{M}_C$ group. The Fitting subgroup $F(G)$ is nilpotent by Lemma \ref{MCFittingnilpotent}. As already mentioned, Corollary \ref{normaldef} implies that $F(G)=dF(G)$. Therefore $dF(G)$ is $\emptyset$-definable in $Th(G)$ by Lemma \ref{defFitting} and $F(G)$ thus coincides with an $\emptyset$-definable nilpotent subgroup of $G$. If $G$ is fcd, then any elementary extension $G'$ of $G$ has the same $c$-dimension. In particular, $G'$ is $\mathfrak{M}_C$, so $F(G')=dF(G')$ by Corollary \ref{normaldef}. Thus the Fitting group and definable Fitting group are both given by the same $\emptyset$-definable formula over $Th(G)$.
\end{proof}

At this point, the following question is natural:

\medskip

%\noindent
{\em Is the solvable analogue of Theorem \ref{defenv} true in an $\mathfrak{M}_C$-group?}

\medskip
\noindent
Our ``normalized'' construction yields the following very partial
answer to this question:

\begin{cor}
Let $G$ be an $\mathfrak{M}_C$ group and $H\leq G$ a solvable subgroup.
If there exist nilpotent subgroups
$A, B\leq H$ such that $A\lhd H$ and $H=AB$, then $H$
is contained in a definable solvable subgroup of $G$.
\end{cor}

\section{Acknowledgements}

The authors thank Alexandre Borovik for his stimulating questions and C\'edric Milliet for a careful reading of an earlier draft.

\end{document}